\documentclass[12pt, twoside, reqno] {amsart}

\usepackage{amsfonts}
\usepackage{amssymb}
\usepackage{latexsym}
\usepackage{epsf,graphicx}
\usepackage{epsf,graphicx,latexsym,%
}

\newcommand{\be} {\begin{eqnarray}}
\newcommand{\ee} {\end{eqnarray}}
\newcommand{\bep} {\begin{eqnarray*}}
\newcommand{\eep} {\end{eqnarray*}}

\newcommand {\inte}{\mathop{\rm Int}\nolimits}

\newcommand {\Hol}{\mathop{\rm Hol}\nolimits}

\newcommand {\Id}{\mathop{\rm Id}\nolimits}

\renewcommand {\Re}{\mathop{\rm Re}\nolimits}
\newcommand {\Ff}{\mathcal{F}}
\newcommand {\A}{\mathcal{A}}
\newcommand {\BB}{\mathcal{B}}

\newcommand{\R}{{\mathbb R}}
\newcommand{\N}{{\mathbb N}}

\newcommand{\C}{{\mathbb C}}

\newcommand {\D}{\mathbb{D}}

\newtheorem{remar}{Remark}[section]
\newtheorem{examp}{Example}[section]
\newtheorem{defin}{Definition}[section]
\newtheorem{corol}{Corollary}[section]
\newtheorem{propo}{Proposition}[section]
\newtheorem{theorem}{Theorem}[section]
\newtheorem{lemma}{Lemma}[section]
\newtheorem{remark}{Remark}[section]

\newcommand{\rema}{\begin{remar}\rm}
\newcommand{\erema}{$\blacktriangleright$\end{remar}}

\newcommand{\exa}{\begin{examp}\rm}
\newcommand{\eexa}{$\blacktriangleright$\end{examp}}

\def\lwvec(#1 #2){\linewd 0.1
           \lvec(#1 #2)
           \linewd 0.05}

\begin{document}

\title[Families of inverse functions]{
Families of inverse functions: coefficient bodies and the Fekete--Szeg\"{o} problem
}

\author[M. Elin]{Mark Elin}

\address{Department of Mathematics,
         Ort Braude College,
         Karmiel 21982,
         Israel}

\email{mark$\_$elin@braude.ac.il}

\author[F. Jacobzon]{Fiana Jacobzon}

\address{Department of Mathematics,
         Ort Braude College,
         Karmiel 21982,
         Israel}

\email{fiana@braude.ac.il}


\begin{abstract}
In this paper we establish the coefficient bodies for a wide class of families of inverse functions. We also completely describe those functions  that provide boundary points of that bodies in small dimensions. As an application we get sharp bounds for Fekete--Szeg\"{o} functionals over some classes of functions defined by quasi-subordination as well as over classes of their inverses. As a biproduct we derive a formula for ordinary Bell polynomials that seems to be new.


\end{abstract}
\maketitle
\section{Introduction}\label{sect-intro}

Estimation of Taylor coefficients for different classes of analytic functions and rigidity problems connected with such estimates are classical problems in Geometric Function Theory. Indeed, if  $\omega,\ \omega(z)=\sum\limits_{k=1}^{\infty}c_kz^k,$ is a holomorphic self-mapping of the open unit disk, then the famous Schwarz Lemma asserts that $|c_1| \leq 1$ and equality is attained only for rotations $\omega(z)=e^{i\theta}z$.
Further, it can be seen that $|c_2| \le 1-|c_1|^2$ and equality is possible only when $\omega$ is the product of the identity mapping $\Id(z)=z$ and an automorphism~of~the~disk. So, besides the estimates, this gives us also special uniqueness (or rigidity) results.

It is well-known that extensions of the above inequalities can be written in a unified form by using the Schur parameters $(\gamma_1,\gamma_2,\ldots)$ of $\omega$. Namely, $|\gamma_j|\le1$ for all $j\ge1$ and if $|\gamma_k|=1$ for some $k$, then $\gamma_j=0$ for all $j>k$ and $\omega$ is a Blaschke product of order $k$; see \cite{Schur, Sim}.

Regarding more complicated situation, many mathematicians studied bounds on coefficients (and coefficient functionals) for various classes of implicit functions, in particular, of inverse functions. Although the coefficient problem for the classes of inverse for all univalent functions and of inverse for starlike functions had been 
solved by Loewner in \cite{Loew} almost one hundred ears ago, the results for other classes are incomplete and attract researchers till nowadays. Among others, it is worth to mention the sharp estimates for early coefficients for inverses of convex functions established in \cite{L-Z} and 
for inverses of starlike functions of positive order obtained in \cite{Ka-Mi}.

Concerning (non-linear) coefficient functionals, the Fekete--Szeg\"{o} functionals are of special interest. They are named so after the seminal work \cite{F-S},  found numerous applications in geometric function theory and were studied by many mathematicians (see, for example, \cite{Ke-Me, Lec, Ma-Mi, Peng}, for general and unified approaches see \cite{ChKSug, Kanas}).
Given a function $f(z)=\sum\limits_{k=0}^{\infty}f_kz^k$ and a number $\lambda \in \C$, consider quadratic functionals of the form
\[
\Phi_n(f, \lambda):=f_{n}f_{n+2}-\lambda f_{n+1}^2, \quad n=0,1,2,\ldots .
\]
The Fekete--Szeg\"{o} problem for some class of analytic functions is to find sharp estimates for the functionals $\Phi_n(\cdot, \lambda)$ over this class.

In particular, $\Phi_n(f):=\Phi_n(f, 1)=\left|\begin{matrix}
                                      f_n & f_{n+1} \\
                                      f_{n+1} & f_{n+2}
                                    \end{matrix}\right| $ is the Hankel determinant of second order. The study of Hankel determinants was initiated by Hayman \cite{Hay} and Pommerenke \cite{Pom66}.
\vspace{3mm}

In this paper we study functions that are not necessarily univalent in the unit disk but are conformal at zero. Section~\ref{sect-Pre} is devoted to needed notations. In particular, we introduce the classes $\A_{\varphi,\psi}$ and $\BB_{\varphi,\psi}$ of two different types: the first one is defined by quasi-subordination in the sense of Robertson \cite{Robertson}, while the second consists of the inverse functions for the elements of the first one. They cover a wide spectrum of various classes.  Also we recall Schur parameters that serve  tools in our study.

In Section~\ref{sect-bell} we consider the Bell polynomials in order to establish relations between Fekete--Szeg\"{o} functionals on analytic functions and their inverses.

 In Section~\ref{sect-main} we establish  recursive formulae for Taylor's coefficients of functions from  $\A_{\varphi,\psi}$ and $\BB_{\varphi,\psi}$ and describe their coefficient bodies. For this aim we follow the line suggested in \cite{Sug} and tested in \cite{MF20}. Partially the results can also be obtained using the famous Lagrange inversion formula although in a more complicated way.

After that we present an algorithm that enables to estimate the Taylor coefficients for functions from  $\A_{\varphi,\psi}$ and $\BB_{\varphi,\psi}$ as well as to prove rigidity properties for these classes and realize it for $n=2$ and $n=3$.

In Section~\ref{sect-app}  we implement the results of the previous sections to solve the Fekete--Szeg\"{o}
problem over some specific $\A_{\varphi, \psi}$ and $\BB_{\varphi, \psi}$. We do not intend to generalize all known results but to demonstrate how our approach works.
For this aim we first choose $\A_{\varphi, \psi}$ to be the class of such functions $F$ those divided by some starlike (or convex) function are of positive real part.
 As a result we provide natural generalizations of some previous results on Fekete--Szeg\"o problem over close-to-convex functions; see, for example, \cite{Ke-Me, K-L-S, Peng}.
Further, we consider the class of functions that (up to shift) subordinate any given function $\psi$.

\medspace

\section{Notions and notations}\label{sect-Pre}
\setcounter{equation}{0}

Let $\D$ be the open unit disk in the complex plane $\C$. We denote the set of holomorphic functions on $\D$ with values in another domain $D\subseteq\C$ by $\Hol(\D,D)$, and by $\Hol(D) := \Hol(D,D)$, the set of all holomorphic self-mappings of $D$. In what follows we use the notion $D_r(c)$ for the open disk of radius $r$ centered at $c \in \C$ and denote $D_r:=D_r(0)$ so that $\D=D_1$.
Denote by $\Omega$ the subclass of $\Hol(\D)$ consisting of functions vanishing at the origin:
\begin{equation}\label{def-U}
\Omega=\{ \omega \in \Hol(\D):\ \omega(0)=0 \}.
\end{equation}

Given $\varphi, \psi \in \Hol(\D,\C)$, consider the class of holomorphic functions
\begin{equation}\label{classA}
\A_{\varphi,\psi}:=\{F \in \Hol(\D,\C): \frac{F}{\varphi}\prec\psi\}.
\end{equation}
In another words, $F \in \A_{\varphi,\psi}$ if there exists a function $\omega \in \Omega$ such that
\begin{equation}\label{subord}
F(z)=\varphi(z)\psi(\omega(z))\quad \text{ for all } z\in \D.
\end{equation}
Recall that the inclusion $F\in\cup_{\varphi\in\Omega}\A_{\varphi,\psi}$ means that $F$ is quasi-subordinate $\psi$ as it was defined by Robertson \cite{Robertson}, $F\prec_q\psi$.
Suppose that  $\varphi, \psi \in \Hol(\D,\C)$ satisfy conditions
\begin{equation}\label{cond}
\varphi(0)=0, \,\, \varphi'(0)\neq 0 \text{ and } \psi(0)\neq 0.
\end{equation}
Then for each $F \in \A_{\varphi,\psi}$ we have $F(0)=0$ and $F'(0)\not=0$, hence $F$ is locally univalent at the origin, so it is invertible and the inverse function $F^{-1}$ satisfies $F^{-1}(0)=0$. Moreover, it can be shown by using the classical Lagrange inversion formula that there is a disk of some radius $r>0$ around zero such that for any $F\in\A_{\varphi,\psi}$ the inverse function $F^{-1}$ is holomorphic in that disk (that is, $r$ depends only on $\varphi$ and $\psi$). This enables us to define  the class
\begin{equation}\label{classB}
\BB_{\varphi,\psi}:=\{F^{-1} : F\in \A_{\varphi,\psi}\}\subset \Hol(D_r,\C).
\end{equation}

Throughout the paper we assume that the  given functions $\varphi, \psi \in \Hol(\D,\C)$ satisfy conditions \eqref{cond}, hence the classes  $\BB_{\varphi,\psi}\subset\Hol(D_r,\C)$ are well-defined.

We now recall the definition of the Schur parameters for a holomorphic self-mapping of the unit~disk.
\begin{defin}[\cite{Schur}, see also \cite{Sim}]\label{Sh-Vec}
Let $\omega \in \Hol(\D)$ be not a constant. Denote by $\sigma$ the mapping acting on  $\Hol(\D)$ and defined by
 \begin{equation}\label{ShurP}
 (\sigma \omega)( z)= \displaystyle\frac{1}{z}\cdot\frac{\omega(z)-\omega(0)}{1-\overline{\omega(0)}\omega(z)}.
 \end{equation}
Consider the sequence $\omega_n=\sigma^{n}\omega$.  The numbers $\gamma_n = \omega_n(0)$, $n=0,1,2,3,\ldots,$  are called the Schur parameters of $\omega.$
\end{defin}

Let now $\omega \in \Omega$ and has  the Taylor expansion $\omega(z) =  c_1z + c_2z^2 + \ldots$,
 and let $\gamma_0, \gamma_1, \ldots$ be  its Schur parameters. Then by Lemma 2.2 in \cite{Sug}, $\gamma_0=c_0=0$ and
\begin{equation}\label{gamma-to-c}
(c_1,\ldots,c_{n})=\overrightarrow{F_n}(\gamma_1,\ldots,\gamma_{n}),
\end{equation}
where the non-analytic polynomial transformation $\overrightarrow{F_n}(\mathbf{z})$   of~$\C^n$ has coordinates defined by the recursive formulas (see \cite{Schur}),
\begin{eqnarray}\label{F}
&& F_1(z_1) =z_1 , \nonumber\\
&& F_m( z_1, z_2,\ldots,z_m) = (1-|z_1|^2)F_{m-1}( z_2, \ldots,z_m)\\
&&\qquad -\overline{z_1}\sum_{k=2}^{m-1}F_{m-k}(z_2, \ldots,z_{m-k+1})F_k( z_1, \ldots,z_k), \quad   m=2,\ldots,n.\nonumber
\end{eqnarray}

In particular, we have
\begin{equation}\label{w-schur}
\omega(0)=0=\gamma_0, \quad  \omega'(0)=\gamma_1\ \text{ and }\  \omega''(0)= 2\gamma_2(1-|\gamma_1|^2).
\end{equation}

Since every function in $\Hol(\D,\C)$ can be identified with the sequence of its Taylor coefficients, the following notations are relevant. Let $\Ff$ be a subclass of $\Hol(\D,\C)$ (or, more generally, holomorphic around zero), denote by $X_n(\Ff)$ the coefficient body of order $n$ for $\Ff$,  that is,
\begin{equation}\label{Coeff-body}
X_n(\Ff)=\left\{(a_0,a_1\ldots,a_n):\ \exists f \in\Ff, \   f(z)=\sum_{k=0}^{n}a_kz^k+o(z^n) \right\}.
\end{equation}
Further, for $j=1,2$ and $n\ge2$, we also use the notion
\begin{equation}\label{*Coeff-body}
X^{(j)}_n(\Ff)=\left\{(a_j,\ldots,a_n):\ \exists f \in\Ff, \   f(z)=\sum_{k=0}^{n}a_kz^k+o(z^n) \right\}.
\end{equation}

\medskip
\section{Bell polynomials}\label{sect-bell}
\setcounter{equation}{0}

We start this section with the following construction that is very familiar in combinatorics and has various applications (see, for example, \cite{C}).
\begin{defin}\label{Def-Bel}
Let $1\leq k\le n$. The ordinary Bell polynomials are given~by
\begin{equation}\label{Bell-ord}
B^o_{n,k}(x_{1},x_{2},\dots ,x_{n-k+1})=k!\sum_{j \in I_k}\prod_{i=1}^{n-k+1} \frac{x_i^{j_{i}} }{j_{i}!},
\end{equation}
where $I_k$ consists of all multi-indexes $j=(j_1,\ldots, j_{n-k+1})$ such that
\begin{equation}\label{I_k}
  j_1,\ldots, j_{n-k+1}\geq0, \quad \sum_{i=1}^{n - k+1} i j_i = n, \quad \sum_{i=1}^{n - k+1} j_i = k .
\end{equation}
\end{defin}
 For instance, setting $k=1,$  we see that $B_{n,1}^o(x) =x_n.$
Similarly, setting $k=n$,  we  get  $B_{n,n}^o(x) = x_1^n$.
The following interesting property of the Bell polynomials will be proved using our further results.
\begin{theorem}\label{th-Bell}
   Let  $\{c_n\}_{n=1}^{\infty}$ be a sequence of complex numbers. Then for any $w\in\C$ and $p\in \N$ we have
\begin{equation*}\label{comb1}
\sum_{k=1}^{p} w^{k-1} B_{p,k}^o\left(c_1,\ldots,c_{p-k+1}\right)=\det \left(\begin{matrix}
          c_1 &       -1 & 0 &  \ldots & 0    \\
          c_2 &    w c_1 & -1 &  \ldots & 0  \\
          c_3 &    w c_2 & w c_1 &  \ldots & 0  \\
           \vdots & \vdots & \vdots  &  {} & \vdots \\
          c_{p-1}&   w c_{p-2} & w c_{p-3} &  \ddots & -1  \\
          c_p&      w c_{p-1} &  w c_{p-2} & \ldots & w c_1
                      \end{matrix}\right).
\end{equation*}
\end{theorem}

The polynomial
\begin{equation}\label{bell-new}
   B_p^o(x_1,\ldots,x_p):=\sum\limits_{k=1}^p B_{p,k}^o\left(x_1,\ldots,x_{p-k+1}\right)
\end{equation}
can be naturally named the complete (ordinary) Bell polynomial. It follows from the proof presented above that $\omega\in\Omega$ if and only if the sequence $\left\{2B^o_p(c_1,\ldots,c_p)\right\}_{p\geq 1}$ consists of the coefficients of a function of the Carath\'{e}odory class. Hence, values of $B_p^o$ satisfy the relations proven in \cite{Sug}. Besides this, taking $w=1$ in Theorem~\ref{th-Bell}, one immediately obtains an explicit formula for $B_p^o(x_1,\ldots,x_n).$

One of the applications of Bell's polynomials is a generalization of the chain rule to higher-order derivatives known as the Fa\`{a}  di Bruno formula. We formulate it for holomorphic functions as it was observed in \cite{MF20}.

\begin{propo}\label{diBruno}
  Let functions $g\in\Hol(\D)$ and $h\in\Hol(\D,\C)$ have the Taylor expansions $g(z)=\sum\limits_{n=1}^\infty a_nz^n$ and $h(z)=\sum\limits_{n=0}^\infty b_nz^n$, respectively. Denote $h\circ g (z)=\sum\limits_{n=0}^\infty c_nz^n$. Then $c_0=h(0)$ and
\begin{equation}\label{ChR1}
c_n = \sum_{k=1}^n b_k B_{n,k}^o \left(a_1,a_2, \ldots, a_{n-k+1}\right)
\end{equation}
for  $n=1,2,3,\ldots$.
\end{propo}

Let $F$ be holomorphic in a neighborhood of zero such that $F(0)=0$ and $F'(0) \neq 0$. Then $G=F^{-1}$ is well defined near zero. Writing
\[
F(z)=\sum_{k=1}^{\infty}a_kz^k\quad  \text{and}\quad  G(z)=\sum_{k=1}^{\infty}b_kz^k,
\]
we have $b_1= \displaystyle \frac{1}{a_1}=:b$.  It turns out that other Taylor coefficients of $G$ can be expressed as follows.

\begin{lemma}\label{lem-main}
Using the above notations we have
\begin{equation}\label{bn}
 b_n=-b^n\sum_{k=1}^{n-1}b_kB_{n,k}^o(a_1,\ldots,a_{n-k+1}),\quad n\ge2.
\end{equation}
Consequently, $\Phi_1(G,\lambda)=-b^6\Phi_1(F,2-\lambda)$ and
$$
\Phi_2(G,\lambda)=b^{8}\Phi_2(F,\lambda)+(4\lambda-5)b^{10}a_2^2\Phi_1(F).
$$
\end{lemma}
\begin{proof}
Since $\left(G\circ F\right)^{(n)}(z)=0$ \ for 
$n\ge2$, we have by Proposition~\ref{diBruno},
\begin{equation*}\label{diBru-cal-b}
0 = \sum_{k=1}^n b_k B_{n,k}^o \left(a_1,a_2, \ldots, a_{n-k+1}\right),\quad n\ge2.
\end{equation*}
Thus
\begin{equation*}\label{diBru-b}
b_na_1^n = -\sum_{k=1}^{n-1} b_k B_{n,k}^o \left(a_1,a_2, \ldots, a_{n-k+1}\right),
\end{equation*}
which gives us \eqref{bn}. In turn, \eqref{bn} implies the rest by straightforward calculations.
\end{proof}

\medskip

\section{Taylor coefficients and rigidity}\label{sect-main}
\setcounter{equation}{0}
Recall that under assumptions \eqref{cond},  for every $F \in \A_{\varphi,\psi}$ there is $G=F^{-1} \in \BB_{\varphi,\psi}\subset \Hol(\D_r,\C)$.

 \subsection{Coefficient body for $\BB_{\varphi,\psi}$}
First we express the Taylor coefficients of $F \in \A_{\varphi,\psi},\displaystyle \ F(z)=\sum_{k=0}^{\infty}a_kz^k$, by the Taylor coefficients of $\varphi$ and $\psi$.

\begin{theorem}\label{Th-main}
Assume that the functions $\varphi,\psi$ have the Taylor expansions $\varphi(z)=\sum\limits_{k=1}^{\infty}\alpha_kz^k$ and $\psi(z)=\sum\limits_{k=0}^{\infty}\beta_kz^k$. Let $F \in\A_{\varphi,\psi}$, that is, $F(z)=\varphi(z)\psi(\omega(z))$ for some $\omega \in \Hol(\D)$ with $\omega(z)=\sum\limits_{k=1}^{\infty}c_kz^k$. Denote $b:= \displaystyle \frac{1}{\alpha_1\beta_0}$. Then $a_1=\frac{1}{b}$ and
\begin{equation}\label{ap}
a_p=\alpha_p\beta_0+\sum_{m=1}^{p-1}\sum_{k=1}^{m}\alpha_{p-m}\beta_kB_{m,k}^o\left(c_1,\ldots,c_{m-k+1}\right),\quad p\ge2,
\end{equation}
where $B^o_{m,k}$ are the ordinary Bell polynomials defined by formula \eqref{Bell-ord}.
\end{theorem}

\begin{proof}
Note that $a_1=F'(0)=\varphi'(0)\psi(0)=\alpha_1\beta_0$. Write the Taylor expansion of the function $\displaystyle\psi\circ \omega(z)=\sum_{k=0}^{\infty}\delta_nz^n$. By Proposition~\ref{diBruno}, $\delta_0=\beta_0$ \nolinebreak and
\begin{equation*}\label{diBruPsi}
\delta_m = \sum_{k=1}^m\beta_k B_{m,k}^o \left(c_1,c_2, \ldots, c_{m-k+1}\right).
\end{equation*}
Now the Leibnitz rule applied to the function $F=\varphi\cdot\left(\psi\circ\omega\right) $ gives
\begin{equation*}\label{leibniz-rule}
\frac{d^p}{dz^p} F(z)=\sum_{m=0}^{p}\binom{p}{m}\varphi ^{(p-m)}(z) \frac{d^m}{dz^m}\left(\psi\circ \omega(z) \right),
\end{equation*}
or equivalently,
\begin{equation*}\label{a_p}
a_p=\sum_{m=0}^{p}\alpha_{p-m}\delta_m=\alpha_p\beta_0+\sum_{m=1}^{p}\alpha_{p-m}\sum_{k=1}^m\beta_k B_{m,k}^o \left(c_1,c_2, \ldots, c_{m-k+1}\right).
\end{equation*}
Since $\alpha_0=0$,  formula \eqref{ap} holds for $p\ge2$.
\end{proof}

Now we are ready to prove Theorem~\ref{th-Bell}.
  \begin{proof}
We prove the theorem for the case $\sum\limits_{n=1}^{p} |c_n|<1$ and $|w|\le1$. Then the assertion in its generality will follow by the uniqueness theorem.

 Under the above restriction $\omega(z)=\sum\limits_{n=1}^{p} c_nz^n \in \Omega$.
The function ${F}(z)=\frac{z}{1-w\omega(z)}$ belongs to $\A_{\varphi,\psi}$ with $\varphi(z)=z,\ \psi(z)= \frac{1}{1-w z}\,$. It follows from Theorem~\ref{Th-main}  that  the Taylor coefficients of $F$ can be calculated by the formula
\begin{equation}\label{a_p-Bell}
a_{p+1}=\sum_{k=1}^{p}w^k B_{p,k}^o\left(c_1,\ldots,c_{p-k+1}\right).
\end{equation}

On the other hand, expanding both sides of the equality ${\frac{F(z)}{z}\left(1-w\omega(z)\right)=1}$ to the Taylor series, we get
\[
\left(1+ a_2z+a_3z^2+ \cdots  \right)   \left( 1-w c_1z-w c_2 z^2+\cdots \right) =1.
\]
Comparison of the the coefficients up to degree $p\in \N$ in the left and right-hand sides yields
\begin{equation}\label{Delta}
\left(\begin{matrix}
                       1 & 0 & 0 & \ldots & 0 \\
                      -w c_1 &1 & 0 &  \ldots & 0 \\
                      -w c_2& -w c_1 & 1 &  \ldots & 0 \\
                       \vdots & \vdots & \vdots  &  \ddots & \vdots \\
                       -w c_{p-1} &  -w c_{p-2} & -w c_{p-3}  & \ldots & 1
                      \end{matrix}\right) \left(\begin{matrix}
                                            a_2 \\
                                            a_3 \\
                                            a_4 \\
                                            \vdots \\
                                            a_{p+1}
                                          \end{matrix}\right)= w\left(\begin{matrix}
                                            c_1 \\
                                            c_2 \\
                                            c_3 \\
                                            \vdots \\
                                            c_{p}
                                          \end{matrix}\right),
\end{equation}
that is, $\{a_2,\ldots,a_{p+1}\}$ is the solution of this linear system.
Note that  the determinant of the matrix in the left-hand side of \eqref{Delta} equals $1$. Denote
\[
\Delta_p: =   \det \left(\begin{matrix}
                 1 & 0 & 0 &  \ldots & 0  &  c_1 \\
                -w c_1 & 1 & 0 &  \ldots & 0 & c_2 \\
                -w c_2 & -w c_1 & 1 &  \ldots & 0 & c_3 \\
                \vdots & \vdots & \vdots & \ddots  &  \vdots & \vdots \\
                 -w c_{p-2} &  -w c_{p-3} & -wc_{p-4} &  \ldots & 1 & c_{p-1} \\
                 -w c_{p-1} &  -w c_{p-2} &   -w c_{p-3}  & \ldots & -w c_1 & c_p
                      \end{matrix}\right).
\]
Then by Cramer's rule we have $a_{p+1}=w\Delta_{p}$. Comparison with \eqref{a_p-Bell} proves our  assertion.
\end{proof}

Theorem~\ref{Th-main} together with Lemma~\ref{lem-main} allow to describe the coefficient bodies of $\A_{\varphi,\psi}$ and $\BB_{\varphi,\psi}$ using an approach similar to that in \cite{Sug}, see also \cite{MF20}. In the next theorem we construct a continuous mapping of  $\overline{\D^{n-1}}$ onto $X^{(2)}_{n}( \BB_{\varphi,\psi})$.
\begin{theorem}\label{Th-main-b}
There exists a transformation $\overrightarrow{\Phi}_n: \D^{n-1} \mapsto X^{(2)}_n( \BB_{\varphi,\psi} )$, of the form $\overrightarrow{\Phi}_n(\mathbf{z})=\overrightarrow{K}_n \circ \overrightarrow{F}_{n-1}({\bf z})$, where $\overrightarrow{K}_n$ is a polynomial transformation of $\C^{n-1}$ and $\overrightarrow{F}_{n-1}$ is defined by~\eqref{F}, such that
\begin{itemize}
  \item [(i)]  $\overrightarrow{\Phi}_n$ is a continuous mapping of  $\overline{\D^{n-1}}$ onto $X^{(2)}_{n}( \BB_{\varphi,\psi})$
  \item [(ii)]  If, in addition, $\psi'(0)\neq 0$, then $\overrightarrow{K}_n$ is an automorphism of $\C^{n-1}$. In this case,  $\overrightarrow{\Phi_n}(\D^{n-1})= \inte X^{(2)}_{n}( \BB_{\varphi,\psi})$,  $\overrightarrow{\Phi}_n (\partial (\D^n))= \partial X^{(2)}_{n}( \BB_{\varphi,\psi})$, and, $\overrightarrow{\Phi}_n$   is a real analytic diffeomorphism on $\D^{n-1}$ but  is not injective on the boundary  of $\D^{n-1}$.
\end{itemize}
\end{theorem}
\begin{proof}
For any $n\geq 2$ consider the polynomial transformation  of $\C^{n-1}$ of the form
$
\overrightarrow{S_n}(\mathbf{z})\left(= \overrightarrow{S_n}(z_1\ldots,z_{n-1})\right):= \left(S_2(z_1),\ldots,S_n(z_1,\ldots,z_{n-1}) \right)
$
defined by $S_p(z_1,\ldots,z_{p-1}):=\alpha_p\beta_0+\sum\limits_{m=1}^{p-1}\sum\limits_{k=1}^{m}\alpha_{p-m}\beta_kB_{m,k}^0\left(z_1, \ldots,z_{m-k+1}\right).$ Then Taylor's coefficients of an arbitrary $F \in \A_{\varphi,\psi}$ with $\omega(z)=\sum\limits_{k=1}^{\infty}c_kz^k$ are $a_1=\frac{1}{b}$ and $a_p=S_p(c_1,\ldots,c_{p-1}),$ $p\ge2$ by Theorem~\ref{Th-main}.

We claim that the transformation $ \overrightarrow{S_n}$ maps the coefficient body $X^{(1)}_{n-1}(\Omega)$ onto $X^{(2)}_{n}(\A_{\varphi,\psi})$.
Moreover,  it is an automorphism of~$\C^{n-1}$ whenever $\psi'(0) \neq 0$.

Indeed, take any $\omega=\sum\limits_{k=1}^{\infty}c_kz^k \in \Omega$. The vector $\overrightarrow{S_n}(c_1,c_2,\ldots,c_{n-1})=:\mathbf{a}=(a_2,\ldots,a_n)$ consists of Taylor coefficients of some $F \in \A_{\varphi,\psi}$; hence $\mathbf{a}\in  X^{(1)}_{n-1}(\A_{\varphi,\psi}) $. Thus $ \overrightarrow{S_n}$ maps $X^{(1)}_{n-1}(\Omega)$ into $X^{(2)}_{n}(\A_{\varphi,\psi})$.
Now, take $\mathbf{a}\in X^{(2)}_{n}(\A_{\varphi,\psi})$. There is a function $F,\ F(z)=bz+a_2z^2+\ldots+a_{n+1}z^{n+1}+o(z^{n+1}),$ which is an element of $\A_{\varphi,\psi}$, that is, $F=\varphi\cdot\psi\circ\omega$  for some $\omega \in \Omega$. Denoting the Taylor coefficients of $\omega$
by $(c_1,c_2,\ldots)$ we see that $\overrightarrow{S_n}(c_1,c_2,\ldots,c_{n-1})=\mathbf{a}.$
Note that $S_p(z_1,\ldots,z_{p-1})- \alpha_1\beta_1z_{p-1}$ depends on $(z_1,\ldots,z_{p-2})$ only. Therefore, since $\alpha_1 \neq 0$ one concludes that if $\beta_1\neq 0$ then $ \overrightarrow{S_n}$ is invertible.

Similarly to the above, consider the polynomial  transformation $\overrightarrow{L_n}$ of $\C^{n-1}$ of the form  $\overrightarrow{L_n}(\mathbf{z}):= \left( L_1(z_1), L_2(z_1,z_2),\ldots,L_{n-1}(z_1,\ldots,z_{n-1}) \right)$ with coordinates
$L_p(z_1,\ldots,z_p)=- b^{p+1}z_p - b^p  \sum \limits_{k=1}^{p-2} L_k(z_1,\ldots,z_{k-1}) B^o_{p,k} \left( z_1,\ldots,z_{p-k+1}\right)$ for $p=1,\ldots,n-1.$  One can see that the mapping $\overrightarrow{L_n}$ is a polynomial automorphism of $\C^{n-1}$ and maps the set $X^{(2)}_n\left( \A_{\varphi,\psi}\right)$ onto the set $X^{(2)}_n\left( \BB_{\varphi,\psi}\right)$.\\
Thus, the polynomial transformation $\overrightarrow{K_n} ({\bf z}):=\overrightarrow{L_n}\circ \overrightarrow{S_n}({\bf z})$ of $\C^{n-1}$ maps $X^{(1)}_{n-1}(\Omega)$ into $X^{(2)}_{n}( \BB_{\varphi,\psi} )$, such that the Taylor coefficients of  $F^{-1} \in \BB_{\varphi,\psi}$ are $b_1=b$ and $b_m=K_m(c_1,\ldots,c_{m-1}),$ $ m\ge2.$ If, in addition, ${\psi'(0)\neq 0}$, then $\overrightarrow{K_n} ({\bf z})$ is automorphism of $\C^{n-1}.$

This implies that other claims of assertion (ii) follow from Proposition~4.2 in \cite{Sug} (cf. the proof of Theorem~4.1 in that work).
\end{proof}

\begin{corol}\label{cor-Phi-FG}
The coefficient body $X_{n}(\BB_{\varphi,\psi})=\left\{\left(0, b\right)\right\} \times X^{(2)}_n( \BB_{\varphi,\psi})$ for $\BB_{\varphi,\psi}$  is a compact set in~$\C^{n+1}$.
\end{corol}

\subsection{Coefficient estimates}\label{CoefB}

A careful look at the proof of Theorem~\ref{Th-main-b} leads to the fact that the last coordinate of the mapping $\overrightarrow{\Phi}_n$ constructed in that theorem  has the form
\[
\Phi_n(\mathbf{z})=H_n(z_1,z_2,\ldots,z_{n-2}) + (1-|z_1|^2)(1-|z_2|^2)\ldots (1-|z_{n-2}|^2)z_{n-1},
\]
where $H_n$ is a (non-analytic) polynomial with respect to its variables and the parameters $(\alpha_1,\ldots,\alpha_n,\beta_0,\ldots,\beta_{n-1})$. Thus
\[
b_n - H_n(\gamma_1,\gamma_2,\ldots,\gamma_{n-2}) = (1-|\gamma_1|^2)(1-|\gamma_2|^2)\ldots (1-|\gamma_{n-2}|^2)\gamma_{n-1},
\]
where $\gamma_1,\gamma_2,\ldots,\gamma_{n-1}$ are the Schur parameters of a function $\omega\in\Omega$. Note that $H_n(\gamma_1,\gamma_2,\ldots,\gamma_{n-2})$ can be expressed through $(b_1,\ldots,b_{n-1})$ explicitly. Then we can obtain estimates on Taylor coefficient and rigidity properties for elements of the class $\BB_{\varphi,\psi}$ by the following algorithm:

\vspace{3mm}

\begin{minipage}{12.5cm}
{\it We already know that $\gamma_0=0$. Assume that $|\gamma_j|<1$ for $j\le n-2$. Then $b_n$ lies in the closed disk centered at $H_n(\gamma_1,\gamma_2,\ldots,\gamma_{n-2})$ and of radius $(1-|\gamma_1|^2)(1-|\gamma_2|^2)\ldots (1-|\gamma_{n-2}|^2)$, where both quantities can be expressed by  $(b_1,\ldots,b_{n-1})$. Moreover, $b_n$ lies on the boundary of the disk if and only if $|\gamma_{n-1}|=1$, that is, $\omega\in\Omega$ is a Blaschke product of order $n-1$. Otherwise,  $|\gamma_{n-1}|<1$ and we then pass to $b_{n+1}$.}
\end{minipage}

\vspace{3mm}

The problem of this approach is that the complicity of the formulas for the disk center and the disk radius quickly increases. Nevertheless, we  demonstrate  how it works for $n=2$ and $n=3$.

\begin{propo}\label{th-b2}
Under the above assumptions the coefficient $b_2$ lies in the closed disk  centered at $c=-\alpha_2\beta_0b^3$ of radius   $r=|\alpha_1\beta_1|\cdot|b|^3$, that is,
\begin{equation}\label{b2-inequality}
\left|b_2+\alpha_2\beta_0b^3 \right| \leq |\alpha_1\beta_1|\cdot|b|^3.
\end{equation}
Consequently, if $\beta_1=0$ then $b_2=   -\alpha_2\beta_0 b^3.$
Otherwise, if $\beta_1\neq0$ but equality in \eqref{b2-inequality} holds, then there exists $\theta \in \R$ such that
\begin{equation*}\label{rigidity11}
F(z) =\varphi(z)\psi(e^{i \theta}z).
\end{equation*}
\end{propo}
\begin{proof}
From formulas \eqref{ap} and \eqref{bn} we get
\begin{equation}\label{b2-calc}
b_2=-b^3(\alpha_2\beta_0+\alpha_1\beta_1\gamma_1).
\end{equation}
 Then
 \begin{equation*}\label{b2-disk}
\left|b_2+\alpha_2\beta_0b^3 \right| = |\alpha_1\beta_1|\cdot|b|^3|\gamma_1|.
\end{equation*}
Since $\gamma_1$ is Schur parameter of some $\omega \in \Omega$, then $|\gamma_1|\leq 1$, which implies~\eqref{b2-inequality}.
Equality in \eqref{b2-inequality} is equivalent to  $|\gamma_1|= 1$.
Burning in mind that $|\gamma_1|=|\omega'(0)|$ (see \eqref{w-schur}), we conclude by the Schwarz Lemma that there exists $\theta \in \R$ such that $\omega(z)=e^{i \theta}z$. The proof is complete.
\end{proof}

We now establish the range of $b_3$  and the corresponding rigidity property. For this purpose we denote
$\widetilde{\alpha}:=-\Phi_1(\varphi,2)=2\alpha_2^2-\alpha_1\alpha_3$  and $\widetilde{\beta}:=-\Phi_0(\psi,2)=2\beta_1^2-\beta_0\beta_2$.

\begin{theorem}\label{th-b3-1}
  Let conditions \eqref{cond} are satisfied and $G\in\BB_{\varphi,\psi},\ G(z)=\sum\limits_{k=1}^{\infty}b_kz^k$. The following assertions hold.
\begin{itemize}
  \item [(I)] If $\beta_1=\psi'(0)\neq0, $ then the coefficient  $b_3$ lies in the closed disk $\overline{D_{\rho_1}(c_1)}$, where
\begin{equation}\label{c-b3}
c_1=b^5\left[ \beta_0^2\widetilde{\alpha}+3\alpha_2\beta_1\gamma_1 +  \alpha_1^2 \widetilde{\beta}\gamma_1^2 \right]
\end{equation}
and
\begin{equation}\label{r-b3}
\rho_1= |b|^4 |\alpha_1\beta_1|\left(1-|\gamma_1|^2\right)
\end{equation}
for some $\displaystyle \gamma_1 \in \D.$
  \item [(II)] If $\beta_1=0$, then $b_3$ lies in the closed disk $\overline{D_{\rho_2}(c_2)}$, where
  \begin{equation}\label{b3in-disk}
 c_2=b^5\beta_0^2\widetilde{\alpha} \quad \text{and} \quad \rho_2=|b|^4\left|\alpha_1\beta_2\right|.
  \end{equation}
\end{itemize}
 \end{theorem}

Note that the parameter $\gamma_1$ in this theorem is the same one as in formula \eqref{b2-calc}. Thus it connects the ranges of $b_2$ and $b_3$. Moreover, a careful consideration of these ranges leads to the following rigidity result which completes Theorem~\ref{th-b3-1}.

\begin{theorem}\label{th-rigidity-Id}
Suppose that the assumptions of Theorem~\ref{th-b3-1} are satisfied. Let $\rho_1,\lambda_1, \rho_2$  and $\lambda_2$ be defined by formulae \eqref{c-b3}--\eqref{b3in-disk}.

If $\beta_1\neq0, $ then $b_3$ lies on the circle  $\partial D_{\rho_1}(c_1)$ if and only if either
 \begin{itemize}
   \item [(i)] there exists $\xi \in \R$ such that $b_2=-(\alpha_2\beta_0+\alpha_1\beta_1e^{i\xi})b^3$; in this case $b_3=c_1$ and  $F(z)=\varphi(z)\psi(\omega(z))$  with  $\omega(z)= e^{i\xi}z$, or
   \item [(ii)] there exists $\theta \in \R$ such that
$
F(z)=\varphi(z)\psi(\omega(z))$  with  $\omega(z)= z   \frac{\gamma_1+ze^{i \theta}}{1+\overline{\gamma_1}ze^{i \theta}}\,.
$
 \end{itemize}

If $\beta_1=\beta_2=0$ then $b_3=  b^5\beta_0^2\widetilde{\alpha}$. If $\beta_1=0,\ \beta_2\neq0,$ then $b_3$ lies on the circle $\partial D_{\rho_2}(c_2)$ if and only if there exists $\theta \in \R$ such that $G$ is the inverse function to $F\in\A_{\varphi,\psi}$ defined by $F(z)=\varphi(z)\psi(e^{i \theta}z)$
and then $
b_3=\widetilde{\alpha}\beta_0^2 b^5- \alpha_1\beta_2 e^{2i\theta} b^4 .
$
\end{theorem}

We prove Theorems~\ref{th-b3-1}--\ref{th-rigidity-Id} simultaneously.
\begin{proof}
From formulas \eqref{ap} and \eqref{bn} we get
\begin{equation}\label{b-form}
  b_3=b^5\left(\beta_0^2\widetilde{\alpha}+\alpha_1\beta_1(3\alpha_2\beta_0\gamma_1-\alpha_1\beta_0\gamma_2(1-|\gamma_1|^2))+\alpha_1^2\gamma_1^2\widetilde{\beta} \right).
\end{equation}
So,
 \begin{equation*}
b_3-b^5 \left(\beta_0^2\widetilde{\alpha}+3\alpha_1\beta_0\beta_1\alpha_2\gamma_1+\alpha_1^2\gamma_1^2\widetilde{\beta} \right)   =
 b^5\alpha_1\beta_0\alpha_1\beta_1(1-|\gamma_1|^2)\gamma_2.
 \end{equation*}
Since $\gamma_2$ is the Schur parameter of some $\omega \in \Omega$, then $|\gamma_2|\leq 1$.   This implies
\begin{equation}\label{b3-disk}
\left|  b_3-b^5 \left(\beta_0^2\widetilde{\alpha}+3b\beta_1\alpha_2\gamma_1+\alpha_1^2\gamma_1^2\widetilde{\beta} \right)\right|  \leq | b^4 | |\alpha_1\beta_1|(1-|\gamma_1|^2).
\end{equation}
Thus, $b_3$ lies in the closed disk $\overline{D_{\rho_1}(c_1)}$, which proves assertion (I) of Theorem~\ref{th-b3-1}.
In other words, either $|b_3-c_1|<\rho_1$, or $b_3$ lies on the circle $\partial D_{\rho_1}(c_1)$.
Equality in \eqref{b3-disk} holds if and only if either $|\gamma_1|=1$ or $|\gamma_2|=1$.

It is well known (see \cite{Sim}) that if $|\gamma_1|=1$ then $\gamma_2=0$. In this case there exists $\xi \in \R$ such that $\gamma_1=e^{i \xi}$. Then from \eqref{b2-calc} we see that
$b_2=-(\alpha_2\beta_0+\alpha_1\beta_1e^{i\xi})b^3$ and $b_3=c_1$, so assertion (i) of Theorem~\ref{th-rigidity-Id} holds.

Otherwise, if $|\gamma_1| < 1$, equality in \eqref{b3-disk} is equivalent to $1=|\gamma_2|=|\omega_2(0)|$. Thus by the Schwarz Lemma there exists $\theta \in \R$ such that $\omega_2(z)\equiv\gamma_2=e^{i \theta}$.

Now using \eqref{ShurP} we can express $\omega_1$ and subsequently $\omega$. This gives  $ \omega(z)= z \cdot \frac{\gamma_1+ze^{i \theta}}{1+\overline{\gamma_1}ze^{i \theta}}$. Thus assertion so assertion (ii) of Theorem~\ref{th-rigidity-Id} follows by definition of $F$ in \eqref{subord}.

In the exceptional case $\beta_1=0$, formula \eqref{b-form} get the form
\begin{equation*}\label{b3-beta10}
b_3=    b^5\left(\beta_0^2\widetilde{\alpha}-\alpha_1^2\gamma_1^2\beta_0\beta_2 \right),
\end{equation*}
which implies
\[
 \left|   b_3-    b^5\beta_0^2\widetilde{\alpha} \right|   = \left|\alpha_1\beta_2b^4\right| | \gamma_1|^2.
\]
Since $\gamma_1$ is the Schur parameter of some $\omega \in \Omega$, hence assertion (II) of Theorem~\ref{th-b3-1} holds. This assertion means  that either $b_3$ lies in the open disk $D_{\rho_2}(c_2)$, or \eqref{b3in-disk} becomes equality. The last is possible only if either $\beta_2=0$ or $|\gamma_1|=1$. If $\beta_2=0$, then $ \left| b_3 - b^5\beta_0^2\widetilde{\alpha}  \right|=0$ which means that $b_3=  b^5\beta_0^2\widetilde{\alpha}$.
If $\beta_2 \neq 0$,  equality in \eqref{b3in-disk} holds only when $1=|\gamma_1|=|\omega'(0)|$. Then, by the Schwarz Lemma, there exists $\theta \in \R$ such that $\omega(z)=e^{i \theta}z$. So $F(z)=\varphi(z)\psi(e^{i \theta}z)$. Substituting $\gamma_1=e^{i\theta}$,  we complete the proof of both theorems.
\end{proof}

\medspace

\medspace

\section{Fekete--Szeg\"o functionals}\label{sect-app}
\setcounter{equation}{0}

Here we demonstrate how the above results can be implemented for estimation of Fekete--Szeg\"{o}  functionals over specific well-known classes of functions. We will use the following auxiliary fact that can be verified directly using basic calculus.

\begin{lemma}\label{lem-aux}
Let $A,B,C$ be real constants. Then for each ${x \in [0,1]}$ we have
\begin{eqnarray*}
&&\max_{0\le x\le1} \left\{A x^2+2B x+C\right\}=\\
&&= \left\{
\begin{array}{ccc}
  C-\frac{B^2}{A}, & \text{ if } A+B<0\, \& \, B>0,\\
  \max(A+2B+C,C), & \text{ otherwise,}
\end{array} \vspace{1mm}
\right.\\
&&\leq \max \left\{A+2B+C,B+C,C \right\}.
\end{eqnarray*}
\end{lemma}

\subsection{Fekete--Szeg\"{o} problem for the case $\mathbf{\psi_\delta(z)=(1-\delta)\frac{1+z}{1-z}+\delta}$}

Let $\delta \in [0,1)$. The function $\psi_\delta$ maps the open unit disk univalently onto the half-plane $\Re w > \delta$.
Clearly, $\beta_0=1$  and  $\beta_j=2(1-\delta)$ for $j>1$. Furthermore, we assume $\alpha_1=1$. In this situation
$a_1=1$ and by formula \eqref{ap},
\begin{equation}\label{ap1}
a_p=\alpha_p +2(1-\delta)\sum_{m=1}^{p-1}  \alpha_{p-m} B_m^o(c_1,\ldots,c_m),\quad p\ge2,
\end{equation}
where the polynomials $B^o_m$ are defined by \eqref{bell-new}. For $F\in\A_{\varphi,\psi_\delta}$ this implies
\[
  \Phi_1(F, \lambda)=\Phi_1(\varphi,\lambda)+2(1-\delta)\left[\alpha_2c_1(1-2\lambda)+c_2+(1-2\lambda(1-\delta))c_1^2\right].
\]
Since for each  $\omega \in \Omega$ we have $|c_2| \leq 1-|c_1|^2$, then
\begin{eqnarray}\label{Phi1a}
 \left|\Phi_1(F, \lambda)\right|&\le& \left|\Phi_1(\varphi,\lambda)\right|+2(1-\delta)+\\
\nonumber +\!\!&2(1-\delta)&\!\!\! \left[|\alpha_2||c_1||1-2\lambda|+ |c_1|^2(|1-2\lambda(1-\delta)|-1)\right]\!.
\end{eqnarray}

\begin{theorem}\label{th-FS_star}
  Let $F\in\A_{\varphi,\psi_\delta}$ for some univalent starlike function $\varphi$. Denote $ \Delta_\delta:=\{\lambda: |1-2\lambda(1-\delta)|+|1-2\lambda |<1\} $.

  Then
  \[
   \left|\Phi_1(F, \lambda)\right|\le \max(1,|4\lambda-3|) +2(1-\delta)\kappa(\lambda),
  \]
where
 \[
\kappa(\lambda)= \left\{ \begin{array}{lcl}
                                                |1-2\lambda(1-\delta)|+2|1-2\lambda|, & \quad & \mbox{if } \  \lambda \notin \Delta_\delta, \vspace{2mm} \\
                1+\frac{|1-2\lambda|^2}{1-|1-2\lambda(1-\delta)|}, & \quad & \mbox{if } \  \lambda \in \Delta_\delta.
                                              \end{array}    \right.
  \]
\end{theorem}
\begin{proof}
By a Theorem~1 in \cite{Ke-Me}, $\left|\Phi_1(\varphi,\lambda)\right|\le \max(1,|4\lambda-3|).$ Further, since $\varphi$ is starlike, $|\alpha_2|\le2$. Thus inequality \eqref{Phi1a} implies
\begin{eqnarray*}
 \left|\Phi_1(F, \lambda)\right|&\le& \max(1,|4\lambda-3|)|+2(1-\delta)+\\
\nonumber +\!\!&2(1-\delta)&\!\!\! \left[2|c_1||1-2\lambda|+ |c_1|^2(|1-2\lambda(1-\delta)|-1)\right]\!.
\end{eqnarray*}
Denote $A=|1-2\lambda(1-\delta)|-1$ and $B=|1-2\lambda|$. Then the result follows by Lemma~\ref{lem-aux}.
\end{proof}

Using Theorem~1 in \cite{Ke-Me} in its complete form, one can easily expand  Theorem~\ref{th-FS_star} for the class
$\A_{\varphi,\psi_\delta}$ where function $\varphi$ is spirallike of arbitrary order.

 Denote by $K_\delta$ the subclass of normalized close-to-convex functions such that
each $f\in K_\delta$ satisfies $\Re\frac{zf'(z)}{\varphi(z)}>\delta$ for a starlike function $\varphi$.
 Thus $F\in\A_{\varphi,\psi_\delta}$ for some starlike $\varphi$  if and only if $F(z)=zf'(z)$ with $f\in K_\delta$. Therefore,
 Theorem~\ref{th-FS_star} immediately imply the following consequence.

 \begin{corol}
   Let $f\in K_\delta$. Then $|\Phi_1(f,\lambda)|\le \max\left(\frac13,|1-\lambda|\right)+\frac{2(1-\delta)}{3}\kappa(\frac{3\lambda}{4}).$
 \end{corol}

In this connection we mention that bounds for the Fekete--Szeg\"o functionals over subclasses of close-to-convex functions
 were studied by many mathematicians (see, for example, \cite{Ke-Me, K-L-S}). In particular, in the last quoted paper the authors
 considered functionals bounds over close-to-convex functions related to convex functions. From this viewpoint the next result seems natural.

\begin{theorem}\label{th-FS_conv}
   If $F\in\A_{\varphi,\psi_0}$ for some univalent convex function $\varphi$, then
 \[
   \left|\Phi_1(F, \lambda)\right|\le \left\{ \begin{array}{lcl}
                                                |\lambda-1| +2 + \frac{|1-2\lambda|^2}{2(1-|1-2\lambda|)}\,,  & \quad & \mbox{if } \  |\lambda-\frac12|<\frac13\,,\vspace{2mm} \\
                                                |\lambda-1|+4|1-2\lambda|, & \quad & \mbox{if } \ |\lambda-\frac12|\ge\frac13\,.
                                              \end{array}    \right.
  \]
\end{theorem}
\begin{proof}
By Corollary~1 in \cite{Ke-Me}, $\left|\Phi_1(\varphi,\lambda)\right|\le \frac{1}{3}+\left( |\lambda-1| -\frac{1}{3}\right)|\alpha_2|^2 
.$ Thus inequality \eqref{Phi1a} implies
\[
 \left|\Phi_1(F, \lambda)\right|\le \frac{1}{3}+\left( |\lambda-1| -\frac{1}{3} \right)|\alpha_2|^2 +2|\alpha_2||c_1||1-2\lambda|+2+ 2 |c_1|^2 (|1-2\lambda|-1).
\]
Further, since $\varphi$ is convex, $|\alpha_2|\le1$ (see \cite{Pom}). Consider the function
$$h(x,y)=\frac{7}{3}+\left( a-\frac{1}{3} \right)y^2+2bxy+2(b-1)x^2,$$
where we denote $x=|c_1|\in [0,1]$, $y=|\alpha_2| \in [0,1]$ and $a=|\lambda-1|$,  $b=|2\lambda-1|$.

It can be easily verified that if $h_x'(x_0,y_0)=h_y'(x_0,y_0)=0$ at some point $(x_0,y_0)\in(0,1)\times (0,1)$, then  $h(x_0,y_0)\le h(0,0)$, hence $h$ attains its maximum on the rectangle boundary. So we are looking for its maxima on each part of  the boundary by using standard calculus. The result is:
\begin{eqnarray*}
&&
   \max h(0,y) = \left\{ \begin{array}{cc}
                             2+a & \mbox{if \ } a\ge\frac13,\vspace{2mm}  \\
                             \frac73 &  \mbox{if \ } a<\frac13,\vspace{2mm}
                           \end{array}  \right.
   \ \  \max h(x,0) = \left\{ \begin{array}{cc}
                             \frac13 +2b  & \mbox{if \ } b\ge 1,  \vspace{2mm}   \\
                             \frac73 &  \mbox{if \ } b<1,\vspace{2mm}
                           \end{array}  \right.\vspace{3mm}  \\
  &&
      \max h(1,y) = a+4b, \qquad   \max h(x,1) = \left\{ \begin{array}{cc}
                            a+4b  & \mbox{if \ } b\ge\frac23, \vspace{2mm}  \\
                            2+a+\frac{b^2}{2(1-b)}  &  \mbox{if \ } b<\frac23\,.
                           \end{array}  \right.
 \end{eqnarray*}

Now we compare these maxima. Let start with $b<\frac23$. One sees that
\begin{eqnarray*}
&&  \left(2+a+\frac{b^2}{2(1-b)}\right)-(2+a) =\frac{b^2}{2(1-b)}\ge 0, \\
&&  \left(2+a+\frac{b^2}{2(1-b)}\right)-(a+4b)=\frac{(3b-2)^2}{2(1-b)}\ge 0.
\end{eqnarray*}
Further, $\left(2+a+\frac{b^2}{2(1-b)}\right)-\frac73 =\frac{1}{2(1-b)}\cdot\left(2(1-b)\left(a-\frac13\right)+b^2\right).$ It is obviously positive when $a\geq \frac13$. Otherwise,  we have $u:=\Re\lambda\in(\frac23,\frac56). $ Then
\[
2(1-b)\left(a-\frac13\right)+b^2\ge (1-2u)^2-2(2-2u)\left(u-\frac23\right)\ge\frac19\,.
\]
Thus the maximum value of $h$ is $2+a+\frac{b^2}{2(1-b)}$ whenever $b<\frac23.$

In the case where $\frac23 \leq b$ we have
$a+4b\ge\max\left(\frac73,2+a, \frac13+2b \right),$ and the proof is complete.
\end{proof}

Theorem~\ref{th-FS_conv} generalizes Theorem~3 in \cite{Peng} where the case of real $\lambda$ was considered.

\begin{remark}\label{rem-PhiF-PhiG1}
 According to Lemma~\ref{lem-main}, Theorems~\ref{th-FS_star}--\ref{th-FS_conv} imply also estimates for the Fekete--Szeg\"{o} functional over the classes $\BB_{\varphi,\psi_0}$ for $\varphi$ being a starlike (respectively, convex) function. Also $|\Phi_1(G)|=|\Phi_1(F)|$.
\end{remark}

\subsection{Fekete--Szeg\"{o} problem for the case $\mathbf{\varphi(z)=z}$}

In this subsection we concentrate on the case $\varphi=\Id$.  This means that $\alpha_1=1$ and $\alpha_j=0$ for $j>1$. Our aim is to estimate the functionals $\Phi_n(\cdot, \lambda),\ n=1,2,$ over the classes $\A_{\Id,\psi}$ and $\BB_{\Id, \psi}$.

\begin{theorem}\label{Th_FS-func1}
Let $F \in\A_{\Id,\psi}$ and $G=F^{-1} \in \BB_{\Id, \psi}$, where $\psi(z)=\sum\limits_{k=0}^{\infty} \beta_kz^k$. Then for every $\lambda\in\C$
 \begin{equation}\label{almostFekte-a}
 \left| \Phi_1(F,\lambda)\right| \leq \max\left( |\beta_0\beta_1|, | \Phi_0(\psi,\lambda) |    \right)
 \end{equation}
and equality holds for the following functions only:
\begin{itemize}
  \item [(i)] if $\beta_1=0$, then for $F(z)=z\psi(ze^{i\theta})$;
  \item [(ii)] if $\beta_0\beta_2=\lambda\beta_1^2$, then for $F(z)=z\psi(z^2e^{i\theta})$;
  \item [(iii)] otherwise, for $F(z)=z\psi\left(z\frac{\gamma_1+\gamma_2z}{1+\overline{\gamma_1}\gamma_2z}\right)$, where $
  \gamma_2= \frac{\left(\beta_0\beta_2-\lambda \beta_1^2\right)\gamma_1^2|\beta_0\beta_1|}{\beta_0\beta_1\left| \left(\beta_0\beta_2-\lambda \beta_1^2\right)\gamma_1^2 \right|}\,$, $\gamma_1\neq0.$
\end{itemize}
\end{theorem}
\begin{proof}
Since $F(z)=z\psi(\omega(z))$, formula \eqref{ap} in Theorem~\ref{Th-main} get a shorter form
\begin{equation}\label{ap-bet1}
a_p=  \sum_{k=1}^{p-1} \beta_kB_{p-1,k}^0\left(c_1,\ldots,c_{p-k}\right),\quad p\ge2,
\end{equation}
where $c_k$, $k=1,2,\dots,$ are the Taylor coefficients of $\omega \in \Omega$. Using \eqref{ap-bet1} and \eqref{gamma-to-c} we have
\begin{equation}\label{Phi1F}
a_1a_3-\lambda a_2^2= \left(\beta_0\beta_2-\lambda \beta_1^2\right)\gamma_1^2+\beta_0\beta_1\gamma_2\left(1-|\gamma_1|^2\right).
\end{equation}
It now follows by the triangle inequality that
\begin{eqnarray}\label{a3-a22}
  \left|a_1a_3-\lambda a_2^2\right| &\le& \left|\beta_0\beta_2-\lambda \beta_1^2\right| |\gamma_1|^2+|\beta_0\beta_1||\gamma_2|\left(1-|\gamma_1|^2\right) \nonumber \\
   &\le&  \left|\beta_0\beta_2-\lambda \beta_1^2\right| |\gamma_1|^2+|\beta_0\beta_1| \left(1-|\gamma_1|^2\right) ,
\end{eqnarray}
which implies $ \left| a_1a_3-\lambda a_2^2 \right| \leq \max\left( |\beta_0\beta_1|, |\beta_0\beta_2-\lambda \beta_1^2 |    \right)$.

Note that the triangle inequality (that is, the first sign in \eqref{a3-a22}) becomes equality only if either one of the summands is zero, or the arguments of both summands are equal. Thus we have the three cases:

  If  $\beta_1=0,$ then the maximum $\max\left( |\beta_0\beta_1|, |\beta_0\beta_2-\lambda \beta_1^2 |  \right) =|\beta_0\beta_2|$ is attained when $|\gamma_1|=1$ and then $\omega(z)= ze^{i\theta}$ (case (i)).

 If  $\beta_0\beta_2-\lambda \beta_1^2=0,$ then the maximum is equal to $|\beta_0\beta_1|$ and is attained when $\gamma_1=0$ and $|\gamma_2|=1.$ By the same considerations as in the proof of Theorems~\ref{th-b3-1}--\ref{th-rigidity-Id}, we conclude that $\omega(z)= z^2e^{i\theta}$ (case (ii)).

 Otherwise, we need that $\arg\left[\left(\beta_0\beta_2-\lambda \beta_1^2\right)\gamma_1^2\right]= \arg\left[ \beta_0 \beta_1 \gamma_2  \right].$ In other words, $\arg\gamma_2=\arg\left[\left(\beta_0\beta_2-\lambda \beta_1^2\right)\gamma_1^2\right] - \arg\beta_0\beta_1.$ In this case, the second sign in \eqref{a3-a22} becomes equality when $|\gamma_2|=1.$ This leads to case (iii). The proof is complete.
\end{proof}

To proceed we notice that the same formulas \eqref{ap-bet1}  and \eqref{gamma-to-c} give us
\begin{eqnarray}\label{a2a4-a3-2}
\Phi_2(F,\lambda)\!\!&=&\!\!a_2a_4-\lambda a_3^2=\\
\nonumber&=&(\beta_1\beta_3-\lambda \beta_2^2)\gamma_1^4+2(1-\lambda)\beta_1\beta_2\gamma_1^2\gamma_2(1-|\gamma_1|^2)+\\
\nonumber \!\!&+&\!\!\beta_1^2(1-|\gamma_1|^2)\left\{(1-|\gamma_2|^2)\gamma_1\gamma_3- \gamma_2^2\left( |\gamma_1|^2+ \lambda(1-|\gamma_1|^2)\right)\right\}.
\end{eqnarray}

Using this relation one can prove the following assertion.

\begin{theorem}\label{Th_FS-func2}
Let assumptions of Theorem~\ref{Th_FS-func1} hold.
Then for every $\lambda\in\C$ we have
 \begin{equation}\label{almost-Fekte-a3}
|\Phi_2(F,\lambda)|\leq \max\left(\left|\Phi_1(\psi,\lambda) \right|, |1-\lambda||\beta_1\beta_2|+\frac{ C(\beta_1, \lambda) }2\,, C(\beta_1, \lambda))
\right).
 \end{equation}
 where $C(\beta_1, \lambda):=|\beta_1|^2\max(1,|\lambda|).$
\end{theorem}
\begin{proof}
  It follows by formula~\eqref{a2a4-a3-2} that $\Phi_2(F,\lambda)$ is not greater than
  \begin{eqnarray*}
 \left| \Phi_1(\psi,\lambda)\right||\gamma_1|^4+2|(1-\lambda)\beta_1\beta_2||\gamma_1|^2(1-|\gamma_1|^2)\\
   + |\beta_1|^2\max\{1,|\lambda|\} (1-|\gamma_1|^2).
\end{eqnarray*}
Denote $C= C(\beta_1, \lambda),\ B=\frac12(2|(1-\lambda)\beta_1\beta_2|-  |\beta_1|^2\max\{1,|\lambda|\})$ and $A= \left| \Phi_1(\psi,\lambda)\right| -2|(1-\lambda)\beta_1\beta_2| $. Then concerning the quadratical polynomial $Ax^2+2Bx+C$ with $x=|\gamma_1|^2$, we get our assertion by Lemma~\ref{lem-aux}.
\end{proof}

\begin{remark}\label{rem-PhiF-PhiG12}
 According to Lemma~\ref{lem-main}, Theorems~\ref{Th_FS-func1}--\ref{Th_FS-func2} imply also estimates for the Fekete--Szeg\"{o} functionals over the classes $\BB_{\Id,\psi}$. Namely,
 \begin{eqnarray}
 \nonumber |\Phi_1(G,\lambda)| & \leq & \frac{1}{|\beta_0|^6} \max \left( |\beta_0\beta_1|, | \Phi_0(\psi,2-\lambda) |   \right); \\
\nonumber |\Phi_2(G,\lambda)|&\leq& \frac{1}{|\beta_0|^8}\cdot|\Phi_2(F,\lambda)|+\frac{|4\lambda-5||\beta_1|^2}{|\beta_0|^{10}}\cdot \max\left( |\beta_0\beta_1|, |\Phi_0(\psi) |    \right).
 \end{eqnarray}
The first estimate is sharp with equality sign only if either one of the conditions (i)--(iii) of Theorem~\ref{Th_FS-func1} holds, while the second one can be improved using Lemma~\ref{lem-main} and formula \eqref{a2a4-a3-2}.
\end{remark}

\begin{theorem}\label{th-Phi2-G}
Let assumptions of Theorem~\ref{Th_FS-func1} hold.
Then for every $\lambda\in\C$ we have
\begin{equation}\label{Phi2G-estim}
 |\Phi_2(G,\lambda)|\leq \left|\frac{\beta_1^2}{\beta_0^8}\right| \cdot \max \left\{C,D,E \right\},
\end{equation}
where $C= \max(1,|\lambda|), \  D = \left| \displaystyle \frac{\Phi_1(\psi,\lambda)}{\beta_1^2}+(4\lambda-5)\frac{\Phi_0(\psi)}{\beta_0^2}\right| $ and \linebreak $E=\displaystyle \frac{1}{2}|4\lambda-5|\left| \frac{\beta_1}{\beta_0}\right| +|1-\lambda|\left|\frac{\beta_2}{\beta_1}\right|+\frac{C}{2}\,.
$
\end{theorem}
\begin{proof}
By Lemma~\ref{lem-main} we have $ \Phi_2(G,\lambda)=b^{8}\Phi_2(F,\lambda)+(4\lambda-5)b^{10}a_2^2\Phi_1(F).$
Now using   formulas \eqref{Phi1F} and \eqref{a2a4-a3-2} and putting ${x:=|\gamma_1|^2 \in [0,1]}$, we estimate $ \Phi_2(G,\lambda)$ on the way similar to the proof of Theorem~\ref{Th_FS-func2}. Then we conclude
\begin{equation*}\label{Phi2G-estim}
\left|\frac{\beta_0^8}{\beta_1^2}\right| |\Phi_2(G,\lambda)|\leq \max_{0\le x\le1} \left\{A x^2+2B x+C\right\},
\end{equation*}
where $ A = \left| \frac{\Phi_1(\psi,\lambda)}{\beta_1^2}+(4\lambda-5)\frac{\Phi_0(\psi)}{\beta_0^2}\right|- |4\lambda-5|\left|\frac{\beta_1}{\beta_0}\right|-2|1-\lambda|\left|\frac{\beta_2}{\beta_1}\right| ,$ \\ $B=\frac{1}{2}|4\lambda-5|\left| \frac{\beta_1}{\beta_0}\right| +|1-\lambda|\left|\frac{\beta_2}{\beta_1}\right|-\frac{1}{2}\max(1,|\lambda|)$ and $ C= \max(1,|\lambda|).$ Denoting $D=B+C$ and $E=A+2B+C$,  we result.
\end{proof}

One can put attention that Theorem~\ref{Th_FS-func2} provides no rigidity result. At the same time, we can do it in some particular cases. In the following example we demonstrate the case where first coefficients of $\psi$ are equal one to another.
\begin{examp}\label{ex1}
 Let $F \in\A_{\Id,\psi}$ and $G=F^{-1} \in \BB_{\Id, \psi}$, where $\psi(z)=\beta_0 + \beta(z+z^2+z^3)+ \sum\limits_{k=4}^{\infty} \beta_kz^k,\  \beta \neq 0$.
Then the Taylor coefficients of the function $G=F^{-1},\ G(z)=z+b_2z^2+b_3z^3+\ldots,$ satisfy the following:
\begin{itemize}
  \item [(1)] By Proposition~\ref{th-b2}, $\left|b_2 \right| \leq \frac{|\beta|}{|\beta_0|^3}$ with equality only for $F(z)=z\psi(e^{i \theta }z)$ with some $\theta \in \R$.

  \item [(2)]  By Theorem~\ref{th-rigidity-Id} and its proof,
  $\left|b_3\right| \leq \frac{|\beta|}{|\beta_0|^5} \max \{|\beta_0|,|2\beta-\beta_0| \}.$
 Moreover, if $\Re \frac{\beta_0}{\beta}\ge1 $,  equality holds only for $F(z)=z\psi(z^2e^{i \xi})$ with some $\xi \in \R$.
Otherwise, equality holds only for $F(z)=z\psi(e^{i \theta }z)$  with some $\theta \in \R$.
\item [(3)]  $\left|\Phi_1(F,\lambda)\right|\le |\beta| \max (|\beta_0|, |\beta_0-\lambda\beta|)$ and $\left|\Phi_1(G,\lambda)\right|\le \frac{1}{\beta_0^6} \left|\Phi_1(F,2-\lambda)\right|$. By Theorem~\ref{Th_FS-func1}, equalities in both estimates hold only if either $\lambda\beta=1$ for $F(z)=z\psi(z^2e^{i\theta})$, or for $F(z)=z\psi\left(z\frac{\gamma+\mu z}{1+\overline{\gamma}\mu z}\right)$, where $\mu= \frac{\left(\beta_0- \lambda \beta\right)\gamma^2|\beta_0|}{\beta_0\left| \left(\beta_0-\lambda \beta\right)\gamma^2 \right|}\,,\ |\gamma|\in(0,1)$.
\item[(4)] By formula \eqref{a2a4-a3-2} and Theorem~\ref{th-Phi2-G} (see also Theorem~\ref{Th_FS-func2}) the following estimates hold
\begin{equation}\label{estimPhi2F-lam0}
\left|\Phi_2(F,\lambda)\right|\leq |\beta|^2\max\left( |\lambda|, |1-\lambda|, 1\right)
\end{equation}
and
\begin{equation*}
 |\Phi_2(G,\lambda)|\leq\left|\frac{\beta^2}{\beta_0^8}\right| \cdot
 \max \left\{ 1,|\lambda|, D,
  \frac{|\beta||4\lambda-5|}{2|\beta_0|}+|1-\lambda|-\frac{\max(1,|\lambda|)}{2} 
 \right\}
\end{equation*}
with $D=\left|\frac{\Phi_1(\psi,\lambda)}{\beta^2}+\frac{(4\lambda-5)\Phi_0(\psi)}{\beta_0^2}\right|$.\\
Moreover,
\begin{itemize}
  \item [(i)] in the case where $\Re \lambda \leq \frac{1}{2}$, equality in \eqref{estimPhi2F-lam0} holds only for $F(z)=z\psi(ze^{i\theta})$;

  \item [(ii)] in the case where $\Re \lambda \geq \frac{1}{2}$, equality in \eqref{estimPhi2F-lam0} holds only for $F(z)=z\psi(z^2e^{i\theta})$.
\end{itemize}
\end{itemize}
\end{examp}

Another application of the above result is: for a given domain $\widetilde D$ containing the image of the quotient  $\frac{F(z)}z$, let $\psi$ be a conformal mapping of the open unit disk $\D$ onto $\widetilde D$, so that $F\in\A_{\Id,\psi}$. Then we can provide ranges for the Taylor coefficients and the Fekete--Szeg\"{o} functionals for inverse functions $G=F^{-1}\in \BB_{\Id,\psi}$.

\begin{examp}\label{ex2}
 Consider all functions $F\in\Hol(\D,\C)$ such that $F(0)=0$,    $F'(0)=1$  and  $\frac{F(z)}z$ does not takes real  negative values. Then
 \[
 \frac{F(z)}z\prec\left(\frac{1+z}{1-z} \right)^2=1+\sum_{n=1}^\infty 4nz^n.
 \]
 Therefore the Taylor coefficients of the function $G=F^{-1},\ G(z)=z+b_2z^2+b_3z^3+\ldots,$ satisfy
\begin{itemize}
  \item [(1)] by Proposition~\ref{th-b2}, $\left|b_2 \right| \leq 4$ and equality holds only when $F(z)=z\left(\frac{1+e^{i \theta }z}{1-e^{i \theta}z}\right)^2$ for some $\theta \in \R$.
  \item [(2)] $\left|b_3-\frac{3}{2}b_2^2 \right| \leq \frac{16-|b_2|^2}{4} $; equality holds only when either $F(z)=z\left(\frac{1+e^{i \theta}z}{1-e^{i \theta}z}\right)^2$ for some $\theta \in \R$, or
      $F(z)=z\left(\frac{1+\gamma z +(z+\overline{\gamma})e^{i \zeta}z}{1-\gamma z -(z-\overline{\gamma})ze^{i \zeta}z}\right)^2$ for some $\gamma \in \D$ and $\zeta \in \R$       by Theorem~\ref{th-rigidity-Id}. Consequently, $\left|b_3 \right| \leq 24$ and equality holds only when $F(z)=z\left(\frac{1+e^{i \theta }z}{1-e^{i \theta}z}\right)^2$ for some $\theta \in \R$.
\end{itemize}

In addition, by Remark~\ref{rem-PhiF-PhiG12}, we have $\left|\Phi_1(G,\lambda)\right|\le 8\max \left(\frac12,|2\lambda-3|\right)$ and by Theorem~\ref{th-Phi2-G},   $|\Phi_2(G,\lambda)|\leq 16 \max \left\{C,D,E \right\},$
where $C= \max(1,|\lambda|),$ $D =\left|43-36\lambda\right|$  and  $E=2|4\lambda-5|  + 2|1-\lambda| +\frac{C}{2}\,.$
\end{examp}

\end{document}